\documentclass[a4paper,12pt,reqno]{amsart}% ------------------------------------------------------------------------------
\newcommand{\cal}{\mathcal}
\usepackage{graphicx}
\usepackage{amsmath}
\usepackage{amssymb}
\usepackage{amsthm}
\usepackage{enumitem}
\usepackage{bbm}
\usepackage[l2tabu,orthodox]{nag}
\usepackage[all,error]{onlyamsmath}
\usepackage[strict]{csquotes}
\usepackage{url}
\newtheorem{theorem}{Theorem}
\newtheorem{lemma}[theorem]{Lemma}
\newtheorem{corollary}[theorem]{Corollary}
\newtheorem{prop}[theorem]{Proposition}
\theoremstyle{definition}
\newtheorem{definition}[theorem]{Definition}
\newtheorem{remark}[theorem]{Remark}

\numberwithin{equation}{section}
\numberwithin{theorem}{section}

\newcommand{\half}{\frac{1}{2}}
\newcommand{\frn}{\frac{n}{2}}
\newcommand{\frp}{\frac{1}{p}}
\newcommand{\fraq}{\frac{1}{q}}
\newcommand{\fro}{\frac{1}{p_0}}
\newcommand{\frqo}{\frac{1}{q_0}}
\newcommand{\eye}{\int\limits_y^1}
\newcommand{\ene}{\mathbb{N}}
\newcommand{\ar}{\mathbb{R}}
\newcommand{\arn}{{\mathbb{R}}^n}
\newcommand{\artwo}{{\mathbb{R}}^2}
\newcommand{\arnn}{{\mathbb{R}}^{4n}}
\newcommand{\ardn}{{\mathbb{R}}^{n}\times{\mathbb{R}}^{n}}
\newcommand{\arnm}{{\mathbb{R}}^{n+m}}
\newcommand{\noi}{\noindent}
\newcommand{\ind}{\indent}
\newcommand{\epsi}{\epsilon}
\newcommand{\egal}{\stackrel{\rm def}{=}}
\newcommand{\parti}{\frac{\partial }{\partial t}}
\newcommand{\partu}{\frac{\partial u}{\partial t}}
\newcommand{\partuu}{\frac{\partial ^2 u}{\partial t^2}}
\newcommand{\partt}{\frac{\partial}{\partial t}}
\newcommand{\bi}{\begin{itemize}}
	\newcommand{\ei}{\end{itemize}}
\newcommand{\be}{\begin{enumerate}}
	\newcommand{\ee}{\end{enumerate}}
\newcommand{\beq}{\begin{equation}}
\newcommand{\eq}{\end{equation}}
\newcommand{\lj}{L^2J^{-2}}
\def\Rn{{{\mathbb R}^n}}
\DeclareMathOperator*{\esssup}{ess\,sup}
\DeclareMathOperator*{\essinf}{ess\,inf}
\DeclareMathOperator*{\sign}{sign\,}

\begin{document}

\title[On the differentiability of the Sobolev norm of $W^{k,1}(M)$ ]{On the differentiability of the Sobolev norm of $W^{k,1}(M)$ on  compact Riemannian manifolds}
\author[D. Cardona]{Duv\'an Cardona}
\address{
  Duv\'an Cardona:
  \endgraf
  Department of Mathematics: Analysis, Logic and Discrete Mathematics
  \endgraf
  Ghent University, Belgium
  \endgraf
  {\it E-mail address} {\rm duvanc306@gmail.com, duvan.cardonasanchez@ugent.be}
  }
  
  \author[J. Delgado]{Julio Delgado*}
\address{
  (*Corresponding Author) Julio Delgado:
  \endgraf
  Departamento de Matem\'aticas
  \endgraf
  Universidad del Valle
  \endgraf
  Cali-Colombia
    \endgraf
    {\it E-mail address} {\rm delgado.julio@correounivalle.edu.co}
  }

  \author[A Munoz]{Andres Felipe Muñoz-Tello}
\address{
   Andres Felipe Muñoz-Tello:
  \endgraf
  Facultad de Ciencias Basicas
  \endgraf
  Universidad Santiago de Cali
  \endgraf
  Cali-Colombia
    \endgraf
    {\it E-mail address} {\rm  andres.munoz00@usc.edu.co}
  }



\date{\today}

\begin{abstract}
    %------------------------- Please type your abstract here ------------------
    In this work we study the differentiability for  the Sobolev norm  of $W^{k,1}(M)$ in the sense of G\^ateaux , where $(M,g)$ is an arbitrary closed Riemannian manifold.
%---------------------------------------------------------------------------
\end{abstract}

\maketitle

\section{Introduction}
%Let $\Omega$ be an open connected subset of $\arn$ we consider the Sobolev space $W^{k,1}(\Omega)$ where $k$ is an integer $\geq 1$.\\
%We study the G\^ateaux differentiability of the norm\\

Let $(M,g)$ be a closed Riemannian manifold and $W^{k,p}(M)$ the corresponding Sobolev space of order $k$ with respect to $L^p(M)$, where  $L^p(M)$ is endowed with  a Riemannian volume form $d\nu(g)$ and $d\nu(g)=\sqrt{\det(g_{ij})}dx,$ and $dx$ stands for the Lebesgue volume element of $\mathbb{R}^n,$ $n=\dim(M)$ (see Definition 2.3).  Herein we find the set of G\^ateaux differentiability for  the Sobolev norm  of $W^{k,1}(M)$. The results also hold for open sets in the euclidean space. We point out  that the case of the spaces $W^{k,p}$, when $1<p<\infty$ boils down to the Fr\'echet differentiability of the $L^p$ norm. The case of our interest  $p=1$ is more pathological. \\

The problem of finding the  set of differentiability of Lipschitz functions was initially considered by Rademacher in \cite{ra}, where he states his classical theorem: 
\begin{center}
	{\it Every Lipschitz function $f:\mathbb{R}^{n}\rightarrow \mathbb{R}$ is differentiable almost everywhere.}
\end{center}
More generally, an extension due to Stepanov shows that the same conclusion holds for locally Lipschitz functions $f:\mathbb{R}^{n}\rightarrow \mathbb{R}^m$. The Rademacher-Stepanov theorem is a fundamental tool in the study of differentiability and is at the basis of some hallmarks, as the Whitney's extension theorem (cf. \cite{wh1}, \cite{wh2}, \cite{wh3}) and its further recent developmets by Fefferman (cf. \cite{fe1}, \cite{fe2}) and Figalli \cite{fi}. The Rademacher-Stepanov theorem can also be deduced from the theory of Sobolev spaces as a consequence of a Calderon's result. Indeed,  if $f\in W_{loc}^{1,\,p}(\arn)$ and $n< p\leq \infty$, then $f$ is differentiable almost everywhere and its derivative equals its weak derivative almost everywhere. Consequently, since every locally Lipschitz function $f$ belongs to $W_{loc}^{1,\,\infty}(\arn)$. Then,   $f$ is differentiable almost everywhere (cf. Chapter 6 of  \cite{ev}.). We would also point out that the Rademacher–Stepanov theorem has been proved in the setting of sub-Riemannian geometry in \cite{Vodopyav}. On the other hand, an important contribution to the differentiability of Sobolev functions with respect to  Sobolev norms has been developed in \cite{golds} and \cite{Vodopyav1}. \\

The converse of Rademacher Theorem  has been recently established by Preiss and Speight (cf. \cite{ps}). Indeed, they prove that: if $n>1$ then there exists a Lebesgue null set in $\mathbb{R}^{n}$ containing a point of differentiability of each Lipschitz function $f: \mathbb{R}^{n} \rightarrow \mathbb{R}^{n-1}$. A more detailed discussion around the 
Rademacher's Theorem and some converse  versions can be found in \cite{preim4}. \\

An extension of Rademacher-Stepanov Theorem to functions defined on separable Banach spaces and valued into Banach spaces was established by Phelps in~\cite{ph}: 

\begin{theorem}[Phelps]\label{phelps}
	Let $E$ be a separable real Banach space and $F$  a real Banach space with  the 	Radon-Nikod\'{y}m property (RNP). If $T: G \rightarrow F$ is a locally Lipschitz mapping  defined on a nonempty open subset  $G$  of $E$. Then $T$ is G\^ateaux differentiable in the complement of a Gaussian null subset of $G$.	
\end{theorem}

Other extensions of Rademacher-Stepanov Theorem to Lipschitz maps between infinite-dimensional Banach spaces have been investigated by several authors e.g. Aronszajn \cite{ar}, Bongiorno~\cite{bon}, Mankiewicz \cite{man}, Phelps~\cite{ph}, Preiss and Zaj\'i\v{c}ek \cite{pz}, among others). The research on the differentiability on Banach spaces has been a very active field  in the last 
decades. We just quote some  few recent examples: In  Potapov and Sukachev's work~\cite{ps14}, they answer to a question in the theory of Schatten- von Neumann ideals of whether their norms have the same differentiability properties as the norms of their commutative counterparts. On the other hand Lindenstrauss et al.~\cite{lin} prove that a real valued Lipschitz function on an Asplund space has points of Fr\'echet differentiability. In ~\cite{gol} Goldys and Gozzi  apply the Phelps Theorem in the proof of existence and uniqueness of solutions of second order parabolic Hamilton-Jacobi-Bellman equations in Hilbert spaces, obtaining optimal feedback for first order stochastic PDEs that arise in economic theory, the theory of population dynamics and financial models. We also point out that the $L^1$-norm has been of recent interest in the analysis of big data due to its major role in minimisations problems \cite{Yimaw}. The study of the set of differentiability of Lipschitz functions on Banach spaces has been carried out for instance in Deville et al.~\cite{de}, Diestel~\cite{di}, Dore and Maleva~\cite{dore} among others. \\
  
We also point out that on compact Riemannian manifolds, the $L^p$ spaces being  totally ordered by inclusion, the $L^1(M)$ space arises as a maximal one with respect to $p$ in the range $1\leq p\leq\infty$ and  the norm $\|\cdot\|_{L^1(M)}$ as a minimal one. In a  future work we will be studying some geometric properties with respect to the $L^1(M)$ norm and  level sets using a geometric characterization as in \cite{CLin}. It is also reasonable to expect some extensions of our results to compact manifolds with boundary and contact manifolds \cite{sas}. Other recent works on Sobolev spaces can be found in  \cite{Agrawaverma}, \cite{Aramaki} and \cite{Zhang}. \\

The main result of this work  establishes the set of G\^ateaux differentiability of the the Sobolev norm  of $W^{k,1}(M)$ and is stated in  Theorem \ref{main1a}, giving also the corresponding formulae for the G\^ateaux derivative. These results extend a $L^1$ version obtained in \cite{detello} for a suitable family of measure spaces. Herein we focous on a version on manifolds thinking on potential applications to differential equations or variational calculus on manifolds. 

%The case of Fr\'echet differentiability for $p=1$  is more pathological. Here****************** we will explicitly find the set of  G\^ateaux differentiability for the  $L^{1}(\Omega, \mu)$ norm, where $(\Omega,\cal{M}, \mu)$ is a measure space. The characterization of this set is given in Theoremark \ref{main1a} of Section 3.

%However the $L^1$ case is pathological. Since $L^1$ is separable, there is not an admissible norm of class $C^1$.****
%Since $L^1$ is separable and $(L^1)'=L^{\infty}$, there is not an admissible norm and $F$-differentiable on $L^1$. Moreover the standard $L^1$ norm is not Fr\'echet differentiable  at any point.... 

%We will prove that the measure for the complement of the differentiability set with the norm of the supremarke in $\ell^{\infty}(\mathbb{R})$ is zero. In addition, for a measure space ($\Omega$,$\mathcal{A},\mu$), we will show that for a norm defined in $L^{\infty}(\Omega)$ space, there is no derivative at any point belonging to this Banach space.

\section[Preliminaries]{Preliminaries}
In this section we briefly recall a basic background. We start with 
 the  classical notions of differentiability on Banach spaces. Secondly, we  introduce our  scale of Sobolev spaces on Riemannian manifolds.
\subsection{G\^ateaux differentiability on Banach spaces}
  The most simple notion of differentiability on Banach spaces is the one of G\^ateaux differentiability. Let $X$, $Y$ be  Banach spaces and $\Omega\subseteq X$ an open set. A function $f:\Omega\rightarrow Y$ is {\it G\^ateaux differentiable ($G$-differentiable)} at a point $a\in\Omega$, if 
\[\lim\limits_{t\rightarrow 0} \frac{f(a+th)-f(a)}{t}=u \left(h\right)\]
exists for all $h\in X$, and $u$ defines  a bounded linear operator from $X$ into $Y$. The operator $u$ is called the G\^{a}teaux derivative of $f$ at the point $a$ and is denoted by $\partial_{G}f(a)$. \\

A function $f:\Omega \rightarrow Y$ is {\it Fr\'echet differentiable ($F$-differentiable)} at $a$, if there is a bounded linear operator $u$ such that
\[f\left(a+h\right)-f\left(a\right)-u\left(h\right)=r\left(h\right)\] where $\lim\limits_{h\rightarrow 0} \frac{\left\|r\left(h\right)\right\|_{Y}}{\left\| h \right\|_{X}}=0$. The operator $u$ is called the Fr\'{e}chet derivative of $f$ at the point $a$  and is denoted by $\partial_{F}f(a)$.\\

It is not difficult to see that, if $f$ is   $F$-differentiable at $a$ then $f$ is $G$-differentiable at $a$ and in that case $\partial_{F}f(a)=\partial_{G}f(a)$.  For a given  function $\varphi:X\rightarrow Y$, we say that $B$ is the set of $G$-differentiability or $F$-differentiability of $\varphi$, if respectively for every element of the set $B$ the function $\varphi$ is $G$-differentiable or $F$-differentiable.\\

On the other hand, a set is called a {\it $G_{\delta}$ set} if it can be expressed as a countable intersection of open sets. This type of set is fundamental in the definition of the {\it F-Asplund  and $G$-Asplund} spaces, since these respectively are Banach spaces in which all convex and continuous function defined in an open and convex subset is $F$-differentiable (resp. $G$-differentiable) in a dense and $G_{\delta}$ set. The following theorem shows  important examples of $G$-Asplund spaces.

\begin{theorem}[Mazur]\label{a4}
	Every separable Banach space is a $G$-Asplund space.
\end{theorem}

Moreover, it is well known that the set of all $F$-Asplund spaces is strictly contained in the set of all $G$-Asplund spaces. In fact, as proved in Deville et al.~\cite{de}, the usual norm of $\ell^{1}(\mathbb{R})$ is not $F$-differentiable at  any point of the space $\ell^{1}(\mathbb{R})$. Then $\ell^{1}(\mathbb{R})$ is not a $F$-Asplund space; but $\ell^{1}(\mathbb{R})$ is a separable Banach space and therefore a $G$-Asplund space by  Mazur's theorem.\\

As we have quoted before, the $L^p$ norm is $F$-differentiable at any point in the case  $1<p<\infty$ and its corresponding derivative is given by \eqref{formF1q}. 
 Moreover, since the dual of $L^{p}(\Omega)$ is $L^{q}(\Omega)$, for $1<p<\infty $ and $\frac{1}{p}+\frac{1}{q}=1$, and they are separable Banach spaces, by  Theorem 2.12 in Phelps~\cite{ph2} the space $L^{p}(\Omega)$ is a $F$-Asplund space for $1<p<\infty$.

%If $p=1$, then the norm of $L^{1} (\Omega)$ is nowhere $F$-differentiable, which leads us to add the remarkark described below. ***********????

\begin{remark} In the special case of the $L^{1}(\Omega,\cal{F}, \mu)$ norm we observe the following regarding  the set of $G$-differentiability. It is well known that(cf. \cite{cohn})  $L^1(\Omega,\cal{F}, \mu)$ is separable provided $\mu$ is $\sigma$-finite and the  $\sigma$-algebra $\cal{F}$ is countably generated. Since $\mathbb{R}$ has the RNP, then the $L^{1}(\Omega, \mu)$ norm satisfies the conditions of  Phelps's Theorem. Therefore, the Gaussian measure of the complement of the $G$-differentiability set is zero and  by Mazur's Theorem this $G$-differentiability set is dense and a $G_{\delta}$ set.	
\end{remark}	

%In the next section we will give an explicit description of the set of $G$-differentiability for the  usual norm of $L^{1}(\Omega, \cal{M}, \mu)$ without assuming  the condition of separability. 

%However, for the pourposes of this manuscript we first present the preliminaries about the scale of Sobolev spaces Riemannian manifolds.

\subsection{Sobolev spaces on Riemannian manifolds}

Let $(M,g)$ be a Riemannian manifold. For any $k\in \mathbb{N}$ and for $u\in C^{\infty}(M),$ $\nabla^{k}u$ denotes the $k$-th
covariant derivative of $u,$ (with $\nabla^{0}u:=u$). In order to illustrate this notion, we observe that in local coordinates the components of $\nabla u$ take the form $(\nabla u)_i=\partial_{i}u,$ while the components of $\nabla^2 u$ in local coordinates are given by
\[ (\nabla^2 u)_{ij}=\partial_{ij}u-\Lambda^{k}_{ij}\partial_k u.  \]
The notation $\Lambda^{k}_{ij}$ stands for the Christoffel symbols with respect to the (unique) Riemannian connection.\\

In general, by definition one has that
\begin{equation}
   |\nabla^k u|^2=g^{i_1j_1}\cdots g^{i_kj_k}(\nabla^k u)_{i_1\cdots i_k}(\nabla^k u)_{j_1\cdots j_k}. 
\end{equation} For every $k\geq 1,$ we denote by $C^{p}_{k}(M)$ the spaces of smooth functions $u\in C^\infty(M),$ such that $|\nabla u|\in L^p(M)$ for any $0\leq j\leq k.$ Hence
\[ 
   C^{p}_{k}(M)=\{u\in C^\infty(M): \int_{M}|\nabla^j u|^pd\nu(g)<\infty,\,0\leq j\leq k\} 
\]where in local coordinates $d\nu(g)=\sqrt{\det(g_{ij})}dx,$ and $dx$ stands for the Lebesgue volume element of $\mathbb{R}^n,$ $n=\dim(M).$ If $M$ is compact one has that $C^{p}_{k}(M)=C^\infty(M).$ Now, with the notation above, one can introduce in a global way the definition of the Sobolev spaces, see Hebey \cite{Hebey}. 
\begin{definition}Let $k\geq 1$ and let $1\leq p<\infty.$ The Sobolev space $W^{k,p}(M,g)$ consists of the completion of $C^{p}_{k}(M)$ with respect to the norm
\begin{equation}
    \Vert u\Vert_{\tilde{W}^{k,p}(M,g)}=\left(\sum_{j=0}^k\int_{M}|\nabla^j u|^pd\nu(g)\right)^{\frac{1}{p}}.
\end{equation}
\end{definition}
In particular, if $(M,g)$ is a compact and without boundary let us follow Taylor \cite[Page 25]{Taylorbook1981}. There is a global formulation for the space $W^{k,p}(M,g)$ in terms of vector fields. In this (exceptional) case, $M$ being a closed manifold, the space $W^{k,p}(M)=W^{k,p}(M,g)$ is independent of the metric $g$ (indeed, if $M$ is not compact, in general, $(M,g)$ depends of $g,$ which can be proved if one considers the case of  Gaussians and non-Gaussian metrics). The fact of the independence of $W^{k,p}(M,g)$ with respect to the choice of the Riemmanian metric $g$ is well-known, but we recall its proof for the convenience of the reader. 
\begin{prop} 
 If $(M,g)$ is a closed Riemmanian manifold, then the Sobolev space $W^{k,p}(M,g)$ is independent of the Riemmanian metric.
  \end{prop}  
\begin{proof} Due to the compactness of $M$, it can be covered by a finite number of charts $(\Omega_m, \phi_m)$ $m=1,2,\dots, N$. Given two Riemmanian metrics $g, \tilde{g}$, there exists a constant $C>1$  such that for any $m$, the components $g_{ij}^m$ of $g$ in $(\Omega_m, \phi_m)$ satisfy  \\ 
\[\frac{1}{C}\tilde{g}_{ij}\leq g_{ij}^k\leq C \tilde{g}_{ij}.\]
We now consider a smooth partition of unity $\eta_m$ subordinate to the covering  $(\Omega_m)$. If  $u\in W^{k,p}(M,\tilde{g})$, we can write $u=\sum_{k=1}^N \eta_{\ell}u$ and
 we have
\begin{align*} \|\eta_{\ell}u\|_{W^{k,p}(M,g)}=& \left(\sum_{j=0}^k\int_{M}|\nabla^j \eta_{\ell}u|^pd\nu(g)\right)^{\frac{1}{p}} \\
=& \left(\sum_{j=0}^k\int_{\phi_{\ell}(\Omega_{\ell})}|\nabla^j\eta_{\ell}u\circ \phi_{\ell}^{-1}|^p\sqrt{det(g_{ir}^{\ell})}d\nu(g)\right)^{\frac{1}{p}} \\
\leq& C_1\left(\sum_{j=0}^k\int_{\phi_{\ell}(\Omega_{\ell})}|\nabla^j\eta_{\ell}u\circ \phi_{\ell}^{-1}|^p\sqrt{det(\tilde{g}_{ir}^{\ell})}d\nu(g)\right)^{\frac{1}{p}} \\
=& C_1\left(\sum_{j=0}^k\int_{M}|\nabla^j\eta_{\ell}u|^pd\nu(g)\right)^{\frac{1}{p}} \\
=&\|\eta_{\ell}u\|_{W^{k,p}(M,\tilde{g})}.
\end{align*}
Hence  $u\in W^{k,p}(M,{g})$. The other inclusion can be proved in a similar way and this completes the proof.
\end{proof}

On the other hand, one can prove that $u\in W^{k,p}(M),$ if and only for every  arbitrary family of smooth vector fields $X_{j_i}\in TM,$ where  $1\leq j_i\leq \ell,$ and $\ell \leq k,$ one has that
\[ X_{j_1}X_{j_2}\cdots X_{j_\ell} u\in L^{p}(M).  \]
In such a case one has the equivalent norm
\begin{equation}
     \Vert u\Vert_{{W}^{k,p}(M)}=\left(\sum_{|\alpha|\leq k}\int_{M}|D^{\alpha} u|^pd\nu(g)\right)^{\frac{1}{p}},
\end{equation}where we have used the multi-index notation 
\begin{equation}\label{alpha}
 D^{\alpha}:=X_{j_1}^{\alpha_1}\cdots X_{j_\ell}^{\alpha_\ell},   
\end{equation}
for an arbitrary choice of $\alpha=(\alpha_1,\cdots , \alpha_\ell).$\\

We recall that in any of the above cases, the elements of Sobolev spaces should be understood as equivalence classes, where $f\sim g$ if $\|f-g\|_{W^{k,1}(M)}=0$.\\

%We point out that a different approach to define Sobolev spaces as the one introduced by Besov and Nikol'skii where the notion of capacity is introduced *****
%gives rise to the same space in the case of closed manifolds. In our case $k$ being a positive integer, the situation is more comfortable. Being in the case of closed manifolds and positive integers $k$ lead us to perhaps the most comfortable situation.

 The notion of capacity may lead to some further properties on the set of G\^ateaux differentiability and we will be addressing this problem in a next work, as well as the study of the set of G\^ateaux differentiability for the norms on Nikol'skii-Besov spaces $B^s_1(M)$. 
\section[The $G$-differentiability set for the $W^{k,1}(M)$ norm]{The $G$-differentiability set for the $W^{k,1}(M)$ norm}
In this section we present our main result. We first recall the formulae for the Fr\'echet differentiability in the case of $L^{p}(\Omega)$ spaces, where $\Omega$ is a bounded domain in $\mathbb{R}^n$ and $1<p<\infty$. Indeed, it is well-known 
 (cf. \cite{di}) that the norm $\|f\|_{L^{p}} = (\int_{\Omega}|f(w)|^{p}dw)^{1/p}$ is Fr\'echet differentiable and moreover its derivative is given by 
\beq\partial_{F}\|f\|_{L^{p}}=\frac{|f|^{p-1}\sign(f)}{\|f\|_{L^{p}}^{p-1}},\label{formF1q}\eq
for any $f\neq 0$. As a consequence, an extension is available for our Sobolev spaces  $W^{k,p}(M)$, with $1<p<\infty$ obtaining that $\|\cdot\|_{W^{k,p}(M)}$ is Fréchet differentiable and sa\-tis\-fies
\[\partial_{F}\|f\|_{W^{k,p}(M)}=\sum_{|\alpha|\leq k}\frac{|D^{\alpha}f|^{p-1}\sign(D^{\alpha}f)}{\|f\|_{W^{k,p}(M)}^{p-1}},\] this for $D^{\alpha}f$, the multi-index notation of a partial derivative of $f$, with $|\alpha|\leq k$ and $f\neq 0$.\\

We can now focus in the pathological case $p=1$ and the Sobolev norm  $\|\cdot\|_{W^{k,1}(M)}$. We start by establishing some preliminary results within a more general context of meausure spaces. Given a measure space $(\Omega, \cal{F}, \mu)$, we denote by $\mathcal{L}^1(\Omega, \mu)$ the space of the real valued integrable functions over $\Omega$.  On $\mathcal{L}(\Omega, \mu)$ one defines the seminorm $\varphi(f)=\int_{\Omega}|f(x)|d\mu(x)$. One also defines on $\mathcal{L}(\mu)$  an equivalence relation $\mathcal{R}$  by, $f\mathcal{R}g\Longleftrightarrow f = g\, \mu-a.e.$ The equivalence class of $f$ is denoted by $\tilde{f}.$ Thus, the quotient $\mathcal{L}(\Omega)/ \mathcal{R}$ defines the classical Banach space $L^{1}(\Omega,\mu)$ endowed with the norm $\varphi(\tilde{f})=\int_{\Omega}|f(x)|d\mu(x)$.\\

The signature function will arise in a natural way in our study of the  G\^ateaux derivative.
\begin{definition} The signature $\sign(x)$ of a real number $x$, is defined by  $\sign(x)=\frac{x}{|x|}$ if $x\neq 0$, and $\sign(x)=0$ if $x=0$. 
\end{definition}

In order to  prove the main theorem we first recall a mild lemma which will help to obtain a pointwise formulae for the  G\^ateaux derivative and its proof can be found in Page 97, Lemma 3.1 of \cite{detello}.
%The proofs are given at the end of the paper.

\begin{lemma}\label{lmain1a} Let $f,\, h:\Omega\rightarrow \ar$  be functions and $x\in\Omega$ such that $h(x)\neq 0$. Then 
\beq\lim\limits_{t\rightarrow 0}\frac{|f(x)+th(x)|-|f(x)|}{t},\label{j6nr}\eq
exists if and only if $f(x)\neq 0$. Moreover, if that is the case
\beq\lim\limits_{t\rightarrow 0}\frac{|f(x)+th(x)|-|f(x)|}{t}=\sign(f(x))h(x).\label{ftja46}\eq
\end{lemma}

We now consider a measure space $(\Omega, \cal{F},\mu)$. As a consequence of the  lemma above we obtain  the following corollary.
\begin{corollary} \label{cordd2} Let $f, h:\Omega\rightarrow \ar$ be measurable functions such that $h(x)\neq 0$ $\mu$-a.e. Then 
\beq\lim\limits_{t\rightarrow 0}\frac{|f(x)+th(x)|-|f(x)|}{t} \mbox{ exists } \mu-a.e\label{limj57}\eq
if and only if $f(x)\neq 0$ $\mu$-a.e.\\

Moreover, if that is the case the formula \eqref{ftja46} holds.
\end{corollary}
\begin{proof} We assume $h:\Omega\rightarrow \ar$ to be a measurable function such that $h(x)\neq 0$ $\mu$-a.e. and we pick $A\in \mathcal{F}$  such that $h(x)\neq 0$ for all $x\in A$, with $\mu(A^c)=0$. We first  assume that  the limit \eqref{limj57} exists $\mu$-a.e. Hence, $f(x)\neq 0$ for every  $x\in A$ by Lemma \ref{lmain1a},  thus  $f(x)\neq 0$  $\mu$-a.e.  \\

On the other hand, if $f(x)\neq 0$  $\mu$-a.e. Then there exists a  set $D\in \cal{F}$ such that $\mu(D^c)=0$ and  $f(x)\neq 0$ for all $x\in D$. By taking $x\in A\cap D$ we have that $f(x), h(x)\neq 0$. Then by Lemma \ref{lmain1a} the limit \eqref{limj57} exists for every $x\in A\cap D$. Therefore the limit \eqref{limj57} 
 exists $\mu$-a.e.\\
 
 The last part of the statement clearly follows from the last part of Lemma \ref{lmain1a} by applying the formula \eqref{ftja46} in a set $\Gamma$ where we have $f(x), h(x)\neq 0$ and $\mu(\Gamma^c)=0.$
\end{proof}
We are now ready to present our main result. Let us fix a compact Riemannian manifold $(M,g).$ 
The following theorem 
gives a precise description of the $G$-differentiability set for the usual  $W^{k,1}(M)$ norm and the corresponding formulae for the  G\^ateaux derivative. We shall use the notation $
 D^{\alpha}:=X_{j_1}^{\alpha_1}\cdots X_{j_\ell}^{\alpha_\ell},$ in \eqref{alpha} for the arbitrary compositions of smooth vector fields. We will denote by $\mu$ the measure corresponding to the Riemannian metric $g$, i.e., $\mu=\nu(g)$.\\

Let $B$ be the subset of $W^{k,1}(M)$ consisting of equivalence classes $\tilde{f}$ such that for some representative $f$ of the  class $\tilde{f}$, one has 
 that $D^{\alpha}f(x)\neq 0\,$ $\mu$ a.e. for every $|\alpha|\leq k$. It is clear  that if $\tilde{f}\in B$, any representative of the class $\tilde{f}$ satisfies the condition defining $B$. That is, if $f_1$ is another representative of the class $\tilde{f}$, then  $D^{\alpha}f_1(x)\neq 0\,$ $\mu$ a.e. for every $|\alpha|\leq k$, due to the definition of the equivalence relationship through the corresponding Sobolev norm of $W^{k,1}(M)$.
\begin{theorem}\label{main1a} Let $M$ be a closed manifold. Let $B$ be the subset of $W^{k,1}(M)$ above defined. Then, the set of $G$-differentiability of the norm $\varphi:=\Vert\cdot\Vert_{W^{k,1}}$ is $B$. Moreover, the G\^ateaux derivative of  $\varphi$ at  $\tilde{f}\in B$ and the direction $h\in W^{k,1}(M)$ is given by 
	\beq\partial_{G}\varphi(\tilde{f})(h)=\sum_{\|\alpha\|_{\mathbb{N}^{n}}\leq k}\left[ \int_{M  } \sign(D^{\alpha}f(x))D^{\alpha}h(x)d\mu(x)\right] ,\label{Gder67}\eq
where $f$ is any representative of the class  $\tilde{f}$ and $D^{\alpha}f$, $D^{\alpha}h$ respectively are the derivatives of $f$ and $h$. 
	\end{theorem}

\begin{proof} %We will present the full proof in the case $n=1$ for simplicity. A careful look at its details will convince the reader ****\\

We assume $\tilde{f}\in B$ and that $f$ is a representative of the class $\tilde{f}$. We start by observing that for any $h\in  W^{k,1}(M)$ the right hand side of the  formula \eqref{Gder67} is independent of the representative in the class $\tilde{f}$. \\
	
We are now going to prove the $G$-differentiability of $\varphi$ at $\tilde{f}$. Let $h\in  W^{k,1}(M)\setminus\{0\}$, an application of the  Lebesgue dominated convergence Theorem gives us
{\setlength\arraycolsep{0pt}
\begin{align*}
	\lim\limits_{t\rightarrow 0}\frac{\varphi(f+th)-\varphi(f)}{t}=& \lim\limits_{t\rightarrow 0} \frac{\sum\limits_{\|\alpha\|_{\mathbb{N}^{n}}\leq k}  \|D^{\alpha}f(x)+tD^{\alpha}h(x)\|_{L^{1}(M)}-\sum\limits_{\|\alpha\|_{\mathbb{N}^{n}}\leq k}  \|D^{\alpha}f(x)\|_{L^{1}(M)}}{t}\\
=&\sum\limits_{\|\alpha\|_{\mathbb{N}^{n}}\leq k}\lim\limits_{t\rightarrow 0}\frac{  \|D^{\alpha}f(x)+tD^{\alpha}h(x)\|_{L^{1}(M)}-  \|D^{\alpha}f(x)\|_{L^{1}(M)}}{t}\\
=&\ \sum\limits_{\|\alpha\|_{\mathbb{N}^{n}}\leq k}\left[  \int\limits_{M}\lim\limits_{t\rightarrow 0}\frac{  |D^{\alpha}f(x)+tD^{\alpha}h(x)|-|D^{\alpha}f(x)|}{t}d\mu(x)\right].
	 	\end{align*}} 
Since $D^{\alpha}h(x)\neq 0$ $\mu$-a.e. and  $D^{\alpha}f(x)\neq 0$ $\mu$-a.e., by applying Corollary \ref{cordd2} we have

	\begin{align*}
\lim\limits_{t\rightarrow 0}\frac{  |D^{\alpha}f(x)+tD^{\alpha}h(x)|-|D^{\alpha}f(x)|}{t}=\sign(D^{\alpha}f(x))D^{\alpha}h(x)
	\end{align*}
$\mu$-a.e., and

 \[\lim\limits_{t\rightarrow 0}\frac{\varphi(f+th)-\varphi(f)}{t}=\sum\limits_{\|\alpha\|_{\mathbb{N}^{n}}\leq k}\left[ \int_{M  } \sign(D^{\alpha}f(x))D^{\alpha}h(x) d\mu(x)\right].\]
	
We now assume that $\tilde{f}\in B^{c}$ and we will show that $\varphi$ does not have $G$-derivative at  $\tilde{f}$. For the set 
\begin{equation}\label{Zeq} Z=\{x\in M:f(x)=0\}\subset B^{c}, 
 \end{equation}
we have
$\mu(Z)>0.$ Since the manifold is compact, we also have $\mu(Z)<\infty$.  We consider $h=1_{Z}$, thus $h\in W^{k,1}(M)\setminus 0$. Since $h=0$ on $Z^c$ and for $\|\alpha\|\neq 0$, $D^{\alpha}h=0$ on $M$, for $t\in\ar$ we have  $|f+th|-|f|=|f|-|f|=0$ on  $Z^c$ and $|D^{\alpha}f+tD^{\alpha}h|-|D^{\alpha}f|=|D^{\alpha}f|-|D^{\alpha}f|=0$ on $M$. Hence, for the calculation of the $G$-derivative of $\varphi$ at $f$ in the direction $h$, is enough to consider the integration on $Z$.
 For $t\neq 0$ we have 
{\setlength\arraycolsep{0pt}
		\begin{align*}
		\frac{\varphi(f+th)-\varphi(f)}{t}=& \sum\limits_{\|\alpha\|_{\mathbb{N}^{n}}\leq k}\frac{  \|D^{\alpha}f(x)+tD^{\alpha}h(x)\|_{L^{1}(M)}-  \|D^{\alpha}f(x)\|_{L^{1}(M)}}{t} \\ 
		 =&\frac{  \|f(x)+th(x)\|_{L^{1}(M)}-  \|f(x)\|_{L^{1}(M)}}{t}\nonumber\\&\quad +\sum\limits_{0<\|\alpha\|_{\mathbb{N}^{n}}\leq k}\frac{  \|D^{\alpha}f(x)+tD^{\alpha}h(x)\|_{L^{1}(M)}-  \|D^{\alpha}f(x)\|_{L^{1}(M)}}{t}\nonumber\\ 
		 =&\frac{ \int_{Z}|t|dx}{t}=\frac{|t|}{t}\mu(Z)=\sign(t)\mu(Z).\nonumber\\
		\end{align*}}
Hence,  we have $(\partial_G^+\varphi)f(h)=\mu(Z)$ and $(\partial_G^+\varphi)f(h)=-\mu(Z)$. Therefore  the $G$-derivative of $\varphi$ at $f$ does not exists in $B^c$, which shows that the set of $G$-differentiability of the norm $\varphi$ is  $B$.\\
 
It is clear that the formula \eqref{Gder67} defines a bounded linear operator from  $ W^{1,1}(M)$ into $\ar$ for every $\tilde{f}$ in $B$, and this concludes the proof of the theorem.
\end{proof}

\begin{remark}  The   Theorem \ref{main1a} also holds for an open subset $\Omega$ of $\arn$ instead of a closed Riemannian manifold $M$. If $\mu(\Omega)<\infty$ the same proof works. If $\mu(\Omega)=\infty$ one needs a slight adaptation in the proof. Indeed, in \eqref{Zeq} one can instead call 
\[Z_1=\{x\in M:f(x)=0\}\subset B^{c}.\] 
 Then  we have
$\mu(Z_1)>0.$ If $\mu(Z_1)=\infty$, $\mu$ being the Lebesgue measure on $\arn$, we can choose a subset $Z$ of $Z_1$ such that $0<\mu(Z)<\infty,$ and the reminder part of the proof follows in the same way with $M=\Omega.$
\end{remark}	
We now deduce a consequence for the mean value theorem for the G\^ateaux derivative, where the knowledge on the set of differentiability as described in Theorem \ref{main1a} will be crucial. 
 
\begin{corollary} If $f, g\in W^{k,1}(M)$,  $D^{\alpha}f(x)\neq 0$ a.e. for every $|\alpha|\leq k$ and assumming that  $\sign(D^{\alpha}f(x))=\sign(D^{\alpha}g(x))$ a.e. for every $|\alpha|\leq k$. Then, $[f,f+g]\subset B$ and  
\[0\leq \Vert f+g\Vert_{W^{k,1}}-\Vert f \Vert_{W^{k,1}}\leq\]\[\leq \sup\limits_{0\leq t\leq 1}\left(\sum_{\|\alpha\|_{\mathbb{N}^{n}}\leq k}\left[ \int_{M  } \sign(D^{\alpha}(f+tg)(x))D^{\alpha}g(x)d\mu(x)\right]\right)\Vert g \Vert_{W^{k,1}}.\]
\end{corollary}
\begin{proof} First, we note that if $0\leq s \leq 1$, we have $(1-s)f+s(f+g)=f+sg$, and $D^{\alpha}f(x)+sD^{\alpha}g(x)\neq 0 \,$ a.e for every  $|\alpha|\leq k$, since 
 $\sign(D^{\alpha}f(x))=\sign(D^{\alpha}g(x))$ a.e. for every $|\alpha|\leq k$.\\

Moreover, $|D^{\alpha}(f+g)(x)|\geq |D^{\alpha}f(x)|$ a.e for every  $|\alpha|\leq k$ again, since $\sign(D^{\alpha}f(x))=\sign(D^{\alpha}g(x))$ a.e. for every $|\alpha|\leq k$. Hence, 
\begin{equation}\label{pos1a}
0\leq \Vert f+g\Vert_{W^{k,1}}-\Vert f \Vert_{W^{k,1}}.
\end{equation}

Therefore, by the mean value theorem for the G\^ateaux derivative, Theorem \ref{main1a} and \eqref{pos1a}, we get
\[0\leq \Vert f+g\Vert_{W^{k,1}}-\Vert f \Vert_{W^{k,1}}=|\Vert f+g\Vert_{W^{k,1}}-\Vert f \Vert_{W^{k,1}}|\leq\]
\[\leq\sup\limits_{0\leq t\leq 1}|\partial_{G}\Vert f+tg \Vert_{W^{k,1}}(g)|\]
\[\leq\sup\limits_{0\leq t\leq 1}\left(\sum_{\|\alpha\|_{\mathbb{N}^{n}}\leq k}\left[ \int_{M  } \sign(D^{\alpha}(f+tg)(x))D^{\alpha}g(x)d\mu(x)\right]\right)\Vert g \Vert_{W^{k,1}}.\]

\end{proof}

\noindent{\bf{Acknowledgements}}\\
\noindent The first author was supported  by the FWO  O\textnormal{d}ysseus  1  grant  G.0H94.18N:  Analysis  and  Partial Differential Equations, by the Methusalem programme of the Ghent University Special Research Fund (BOF)
(Grant number 01M01021). The second  author was supported by Vic. Inv. Universidad del Valle CI 71352.\\

%pointed out a number of corrections and valuable suggestions helping to improve the %main results and the presentation of the manuscript.

\noindent $\bullet$ Our manuscript has not associated data.\\

\noindent $\bullet$ No conflict of interest/Competing interests

%\section*{References}

%\bibliography{mybibfile}

\end{document}